\documentclass[11pt,reqno]{amsart}
\usepackage{amssymb,latexsym}
\usepackage{graphicx}
\usepackage{fancyhdr,amssymb}
\usepackage{url}		
\usepackage[margin=10pt,small,labelfont=bf,hang]{caption}
\usepackage{textcomp}
\usepackage{amsmath}
\usepackage{array}
\usepackage{algorithm}
\usepackage[noend]{algpseudocode}
\usepackage{color}
\usepackage{wrapfig}
\usepackage{booktabs}
\usepackage{cite}
\usepackage{multicol}
\usepackage{mathtools}

\theoremstyle{plain}
\numberwithin{equation}{section}
\newtheorem{thm}{Theorem}[section]
\newtheorem{theorem}[thm]{Theorem}
\newtheorem{lemma}[thm]{Lemma}

\newtheorem{definitions}[thm]{Definitions}

\newtheorem{corollary}[thm]{Corollary}

\begin{document}
\fancyhead{}
\renewcommand{\headrulewidth}{0pt}
\fancyfoot{}
\fancyfoot[LE,RO]{\medskip \thepage}
\fancyfoot[LO]{\medskip MONTH YEAR}
\fancyfoot[RE]{\medskip VOLUME , NUMBER }

\setcounter{page}{1}

\title[Infinite]{notes on approximate golden spirals with whirling squares} 
\author{Carlos Silveira}
\address{Department of Biotechnology\\
                Federal University of Paraiba - UFPB - Brazil\\}
\email{carlos.silveira@cbiotec.ufpb.br\\}

\begin{abstract}
We have extended some known results of the approximate golden spirals to generalized m-spirals built with whirling squares for any $m$ ratio ($m>1$). In particular, we have proved that circumscribed circles around squares intercept the pole. This implies that poles of any m-spirals lie on a circumscribed circle about the first square. We show here that these and other fascinating properties are not exclusive of golden rectangles. The classical golden constructors may not be alone.
\end{abstract}

\maketitle

\section{Introduction}
\vspace{0.1cm}

Consider a line segment $\overline{A B}$ and a sectioning point $C$. The new segments $\overline{A C}$ and $\overline{C B}$ can express a ratio $m$ such that: 

\begin{equation}\label{eq1}
m = \frac{AB}{AC}
\vspace{0.1cm}
\end{equation}

Undoubtedly, the most famous section of a segment is the golden section, known since ancient Greeks, and commonly defined by proportion:

\begin{equation}\label{eq3}
\frac{AB}{AC}=\frac{AC}{CB}=\phi
\vspace{0.1cm}
\end{equation}

where $\phi = 1.618033...$ is called golden or Fibonacci number \cite{Livio_a}. Stakhov \cite{Stakhov_a} generalized the golden section in p-golden sections (or p-sections), defined by:

\begin{equation}\label{eq7}
  \begin{aligned}
    &\left(\frac{AB}{AC}\right)^p =\frac{AC}{CB}, &where \hspace{0.2cm}&
    \frac{AB}{AC} =\alpha_{p}, \forall \mspace{3mu} p \in \mathbb{N}&
   \end{aligned}
\vspace{0.1cm}
\end{equation}

\begin{wraptable}{r}{50mm}
  \begin{center}
  \footnotesize
  \renewcommand{\arraystretch}{1.0}
  \vspace{-0.25cm}
  \begin{tabular}{|c|c|}
  \hline
  p & $\alpha_{p}$ \\ 
  \hline
  0 & 2.0000000000... \\ 
  \hline
  1 & 1.6180339887... \\ 
  \hline
  2 & 1.4655712318... \\ 
  \hline
  3 & 1.3802775690...\\
  \hline
  4 & 1.3247179572...\\
  \hline
  5 & 1.2851990332...\\
  \hline
  ... & ... \\ 
  \hline
  $\infty$ & 1.0000000000...\\
  \hline
  \end{tabular}
  \vspace{0.2cm}
  \caption{p-Fibonacci numbers}
  \label{tab1}
  \end{center}
  \vspace{-20pt}
\end{wraptable}

The $\alpha_{p}$ are p-Fibonacci numbers (see table \ref{tab1}). Note that for $p=1$ we have the classical golden proportion. Exploring the limits of p-Finobacci equations, we can see that the values of $p$ from $0$ to $\infty$ gives $\alpha_{p}$ in the range $\left(2,1\right[$. This is the same interval when $AC$ tends to midpoint $AM$ ($m\to2$), and $AC$ tends to $AB$ ($m\to1$). Others attempts to generalize the golden proportion can be found in \cite{Fowler_a,Falbo_a}.

Logarithmic spirals have the property that the distance between turns increases in geometric progression. Golden (logarithmic) spirals have $\phi$ as their growth factor. It is possible to approximate golden spirals with whirling squares (whose sides are in successive $\phi$ ratios), through quarter-circle arcs inscribed into squares to simulate the turns (figure \ref{Fig4a}, for $p=1$). Such squares can tile a rectangle also called golden, because its sides will be in golden ratio \cite{Livio_a}.  

Several interesting properties involving the golden rectangle and the approximate golden spiral have been described
\cite{Brousseau_a,Holden_a,Hoggartt_a,Pickover_a}. We show here that many of these properties can be extended for any whirling square and its correlated pseudo-logarithmic spiral built with any $m$ ratio, what henceforth we call m-spiral (figure \ref{Fig4a}). We are particularly interested in the location of poles in such m-spirals. 

\begin{figure}[h]
\centering
\includegraphics[trim = 0mm 18mm 0mm 0mm, clip, width=0.8\textwidth]{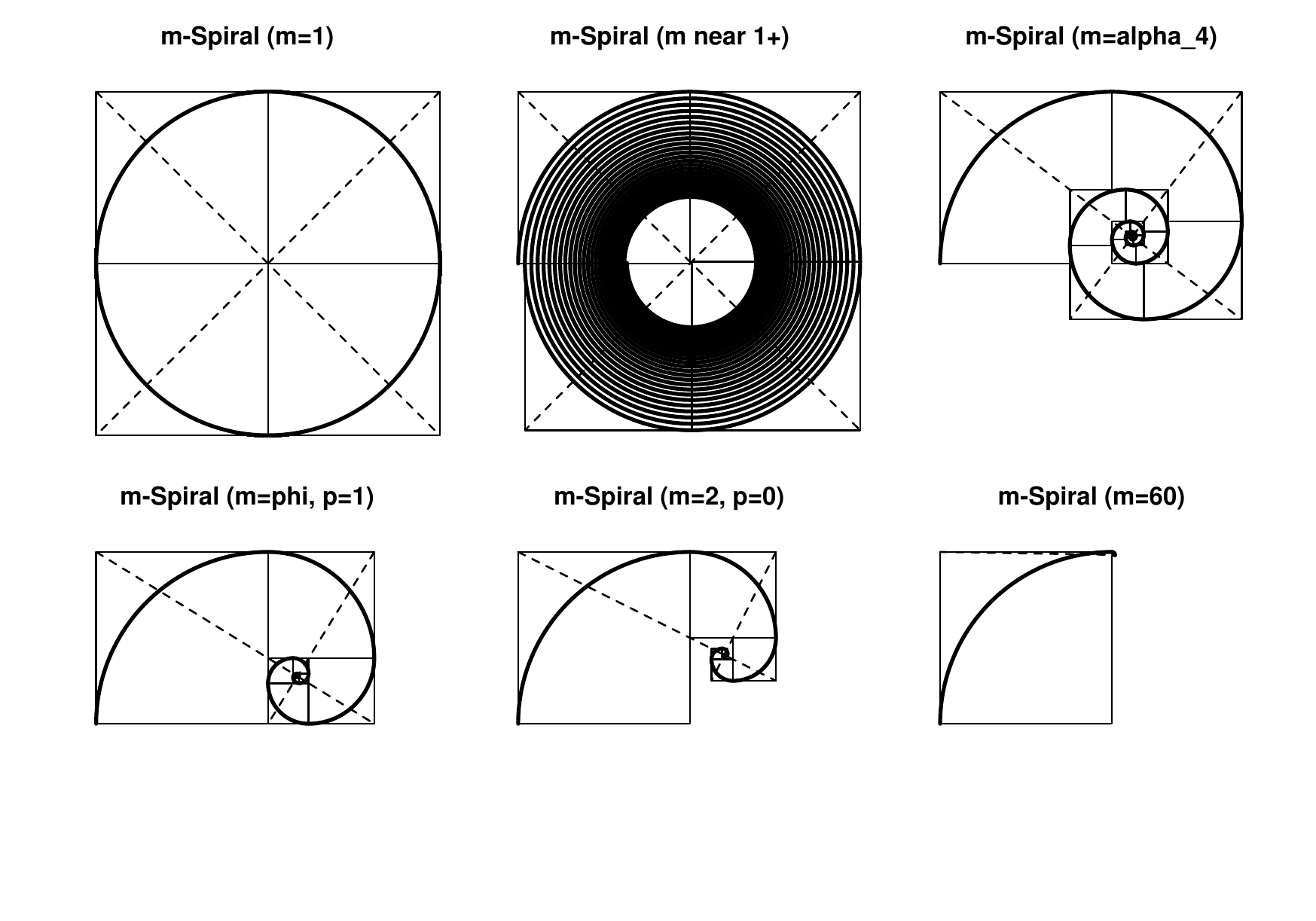}
\caption{Some m-spirals. Dashed lines are diagonals}
\label{Fig4a}
\end{figure}
\vspace{0.0cm}

\section{Notes on poles}

\begin{definitions}\label{definB}
Let m-spirals be a generalized whirling square spiral whose square sides $L_i$ are in successive $m$ ratio $(m=\frac{L_{i}}{L_{i-1}})$, $\forall m \in \mathbb{R}, m>1$ and $\forall i \in \mathbb{N},i \geq0$. Let $L=L_0$ and $(X_{0},Y_{0})$ be the side length and coordinate center of first square, respectively. Let $(X_{eye},Y_{eye})$ be coordinates of the spiral pole. Let $(X_{ur},Y_{ur})=(X_{0}+\frac{L}{2},Y_{0}+\frac{L}{2})$ and $(X_{lr},Y_{lr})=(X_{0}+\frac{L}{2},Y_{0}-\frac{L}{2})$ be coordinates of upper and lower right vertices of the first square, respectively. 
\end{definitions}

\begin{lemma}\label{lemmaA}
 For any m-spiral, the center of $i^{th}$ square can be found at:
 
\begin{equation}\label{eq08}
  \begin{dcases}
    &X_{i}=X_{ur}+L\left(k_1(i)\frac{m-1}{m^2+1}-(-1)^ik_2(i)\right)\\[0.5em]
    &Y_{i}=Y_{ur}-L\left(k_1(i)\frac{m+1}{m^2+1}+k_2(i)\right)
  \end{dcases}
\end{equation}

\begin{equation}\label{eq09}
  \begin{aligned}
    &where\\[0.2em]
    &k_1(i)=1-\left(-\frac{1}{m^2}\right)^{\left\lfloor{\frac{i}{2}}\right\rfloor}&
    k_2(i)=(-1)^{\left\lfloor{\frac{i}{2}}\right\rfloor}\frac{1}{2m^i}
  \end{aligned}
\end{equation} 
\vspace{0.4cm}

\end{lemma}

\begin{proof}
 
Algorithm 1 presents a procedure to generate the m-spirals of Figure 1. It was implemented using a finite automaton with one initial state and four iterative states (see figure \ref{Fig5a}). The convergence to pole is made by recursive searching of square centers in clockwise direction.

It is not difficult to see that function \textbf{SearchPole$()$} produces the series: 

  \begin{equation}\label{eq10}
    \begin{dcases}
      &X_{i} = X_{0}+\frac{L}{2}+L\left(\frac{1}{m}-\frac{1}{m^2}-\frac{1}{m^3}+\frac{1}{m^4}+\frac{1}{m^5}-\cdots \pm \frac{1}{m^{(i-1)}} \pm \frac{1}{2 m^{(i)}}\right)\\[0.5em]
      &Y_{i} = Y_{0}+\frac{L}{2}-L\left(\frac{1}{m}+\frac{1}{m^2}-\frac{1}{m^3}-\frac{1}{m^4}+\frac{1}{m^5}+\cdots \pm \frac{1}{m^{(i-1)}} \pm \frac{1}{2 m^{(i)}}\right)\\
    \end{dcases}
  \end{equation}
  \vspace{0.4cm}

We can strategically rearrange these series, decomposing them in parts, in order to distinct two geometric series: one with first term $a_{1}=\frac{1}{m}$ and other with $a_{1}=\pm\frac{1}{m^2}$, but both with ratio $r=-\frac{1}{m^2}$.

  \begin{equation*}
    \begin{dcases}
      &X_{i} = X_{ur}+L\left(\boxed{\frac{1}{m}-\frac{1}{m^3}+\frac{1}{m^5}-\cdots}\cdots \boxed{-\frac{1}{m^2}+\frac{1}{m^4}-\frac{1}{m^6}+\cdots}\cdots \boxed{\pm\frac{1}{2m^{(i)}}}\right)\\[1.0em]
      &Y_{i} = Y_{ur}-L\left(\boxed{\frac{1}{m}-\frac{1}{m^3}+\frac{1}{m^5}-\cdots}\cdots \boxed{+\frac{1}{m^2}-\frac{1}{m^4}+\frac{1}{m^6}-\cdots}\cdots \boxed{\pm\frac{1}{2m^{(i)}}} \right)\\
    \end{dcases}
  \end{equation*}
  \vspace{0.4cm} 

We have to deal with an asymmetry when comparing even and odd $i^{th}$ square centers. This is best perceived by examples. See below:

  \begin{equation*}
    \begin{dcases}
      &X_{5} = X_{0}+\frac{L}{2}+L\left(\frac{1}{m}-\boxed{\frac{1}{m^2}}-\frac{1}{m^3}+\boxed{\frac{1}{m^4}}+\frac{1}{2m^5}\right)\\[1.0em]
      &Y_{5} = Y_{0}+\frac{L}{2}-L\left(\frac{1}{m}+\boxed{\frac{1}{m^2}}-\frac{1}{m^3}-\boxed{\frac{1}{m^4}}+\frac{1}{2m^5}\right)\\
    \end{dcases}
  \end{equation*}
  \vspace{0.2cm}

  \begin{equation*}
    \begin{dcases}
      &X_{6} = X_{0}+\frac{L}{2}+L\left(\frac{1}{m}-\boxed{\frac{1}{m^2}}-\frac{1}{m^3}+\boxed{\frac{1}{m^4}}+\frac{1}{m^5}-\frac{1}{2 m^{6}}\right)\\[1.0em]
      &Y_{6} = Y_{0}+\frac{L}{2}-L\left(\frac{1}{m}+\boxed{\frac{1}{m^2}}-\frac{1}{m^3}-\boxed{\frac{1}{m^4}}+\frac{1}{m^5}+\frac{1}{2m^{6}}\right)\\
    \end{dcases}
  \end{equation*}
  \vspace{0.4cm}

If we ignore the last term $\pm\frac{1}{2m^i}$, in even square centers (like $(X_6,Y_6)$), the geometric series with even exponents have one term less than odd ones. We can facilitate demonstration process if we introduce the correction above, aiming equal the parity of terms:
 
 \begin{equation*}
  \begin{dcases}
    &X_{6} = X_{0}+\frac{L}{2}+L\left(\frac{1}{m}-\boxed{\frac{1}{m^2}}-\frac{1}{m^3}+\boxed{\frac{1}{m^4}}+\frac{1}{m^5}-\boxed{\frac{1}{m^6}}+\frac{1}{2 m^{6}}\right)\\[0.5em]
    &Y_{6} = Y_{0}+\frac{L}{2}-L\left(\frac{1}{m}+\boxed{\frac{1}{m^2}}-\frac{1}{m^3}-\boxed{\frac{1}{m^4}}+\frac{1}{m^5}+\boxed{\frac{1}{m^6}}-\frac{1}{2m^{6}}\right)\\
  \end{dcases}
\end{equation*}
\vspace{0.4cm}

Be $S_1(n)$ and $S_2(n)$ the geometric series for $(a_1=\frac{1}{m},r=-\frac{1}{m^2})$ and $(a_1=-\frac{1}{m^2},r=-\frac{1}{m^2})$, respectively. So, for even squares, equations \eqref{eq10} can be rewritten as ($\forall j \in \mathbb{N},j \geq0$):

  \begin{equation*}
    \begin{dcases}
      &X_{2j} = X_{ur}+L\left(S_1(j)+S_2(j)-\frac{(-1)^j}{2m^{(2j)}}\right)\\[0.5em]
      &Y_{2j} = Y_{ur}-L\left(S_1(j)-S_2(j)+\frac{(-1)^j}{2m^{(2j)}}\right)
    \end{dcases}
  \end{equation*}
  \vspace{0.4cm}

For odd squares:

  \begin{equation*}
    \begin{dcases}  
      &X_{2j+1} = X_{ur}+L\left(S_1(j)+S_2(j)+\frac{(-1)^j}{2m^{(2j+1)}}\right)\\[0.5em]
      &Y_{2j+1} = Y_{ur}-L\left(S_1(j)-S_2(j)+\frac{(-1)^j}{2m^{(2j+1)}}\right)
    \end{dcases}
  \end{equation*}
  \vspace{0.4cm}	

Given that $S(n)=(1-r^n)\frac{a_1}{1-r}$ and making $k(n)=1-r^n$, for even squares we have:

  \begin{equation*}
    \begin{dcases}
      &X_{2j} = X_{ur}+L\left(k(j)\frac{m-1}{m^2+1}-\frac{(-1)^j}{2m^{(2j)}}\right)\\[0.5em]
      &Y_{2j} = Y_{ur}-L\left(k(j)\frac{m+1}{m^2+1}+\frac{(-1)^j}{2m^{(2j)}}\right)
    \end{dcases}
  \end{equation*}
  \vspace{0.4cm}

For odd squares:

  \begin{equation*}
    \begin{dcases}
      &X_{2j+1} = X_{ur}+L\left(k(j)\frac{m-1}{m^2+1}+\frac{(-1)^j}{2m^{(2j+1)}}\right)\\[0.5em]
      &Y_{2j+1} = Y_{ur}-L\left(k(j)\frac{m+1}{m^2+1}+\frac{(-1)^j}{2m^{(2j+1)}}\right)
    \end{dcases}
  \end{equation*}
  \vspace{0.4cm}

We must now unify even and odd square equations into a single expression for a generic $i \geq 0$. We need an operator that when applied over $\frac{2j}{2}$ and $\frac{(2j+1)}{2}$ retrieves $j$. It is simple to verify that \emph{floor} operator $\lfloor x \rfloor$ does it. Therefore, we can conceive the following functions:

\begin{equation*}
  \begin{aligned}
    K_{1}(i) &= 1-\left(-\frac{1}{m^2}\right)^{\left\lfloor{\frac{i}{2}}\right\rfloor} &&
    K_{2}(i) &= \frac{(-1)^{\left\lfloor{\frac{i}{2}}\right\rfloor}}{2m^i}
  \end{aligned}
\end{equation*}
\vspace{0.4cm}

and state that:

\begin{equation*}
  \begin{dcases}
    &X_{i} = X_{ur}+L\left(k_1(i)\frac{m-1}{m^2+1}-(-1)^ik_2(i)\right)\\[0.5em]
    &Y_{i} = Y_{ur}-L\left(k(j)\frac{m+1}{m^2+1}+k_2(i)\right)
  \end{dcases}
\end{equation*}
\vspace{0.4cm}

\end{proof}

\begin{figure}[h]
\centering
\includegraphics[width=0.8\textwidth]{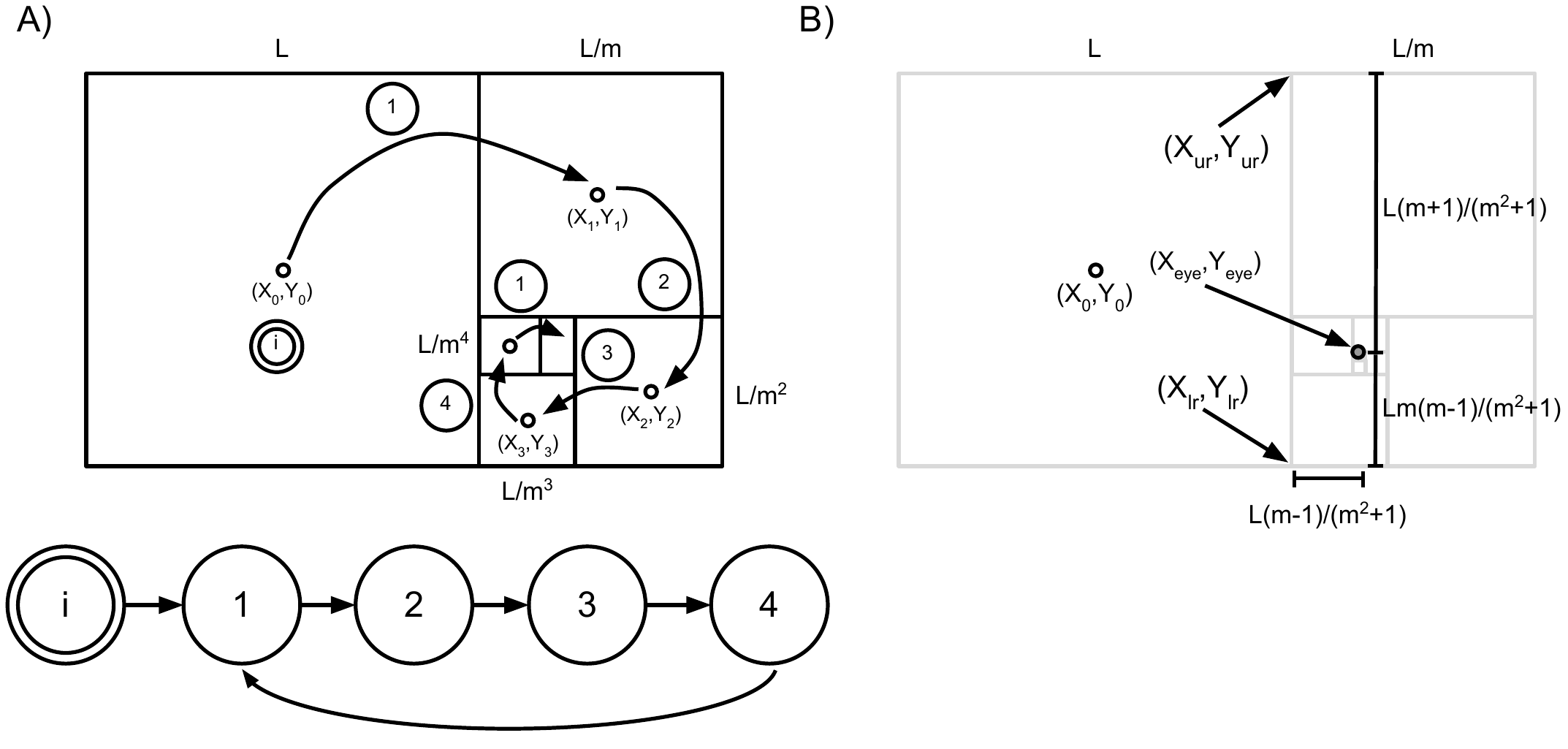}
\caption{A) Finite automaton for m-spiral generation. B) Geometric interpretation of theorem \ref{theoremB}.}
\label{Fig5a}
\end{figure}
\vspace{0.4cm}

\begin{algorithm}[h]
  \caption{Spiral Pole Determination (pseudocode)}
  \scriptsize \begin{algorithmic}[1]
  \label{alg1}
  \Function{SearchPole}{$S,L,m,X_{eye},Y_{eye},i,n$}
  \Comment{$S:State$,$L:side$,$(X_{eye},Y_{eye}):Pole$,$i:counter$,$n:iteractions$}
    \If {$(i>n)$} \textbf{return} $(X_{eye},Y_{eye})$ \EndIf
    \If {$(S==1)$}
      \State {$X_i=X_{eye}+\frac{L}{2}+\frac{L}{2m}$}
      \State {$Y_i=Y_{eye}+\frac{L}{2}-\frac{L}{2m}$}
      \State {SearchPole$(S=2,L=\frac{L}{m},m,X_{eye}=X_i,Y_{eye}=Y_i,i=i+1,n)$}
    \EndIf
    \If {$(S==2)$}
      \State {$X_i=X_{eye}+\frac{L}{2}-\frac{L}{2m}$}
      \State {$Y_i=Y_{eye}-\frac{L}{2}-\frac{L}{2m}$}
      \State {SearchPole$(S=3,L=\frac{L}{m},m,X_{eye}=X_i,Y_{eye}=Y_i,i=i+1,n)$}
    \EndIf
     \If {$(S==3)$}
      \State {$X_i=X_{eye}-\frac{L}{2}-\frac{L}{2m}$}
      \State {$Y_i=Y_{eye}-\frac{L}{2}+\frac{L}{2m}$}
      \State {SearchPole$(S=4,L=\frac{L}{m},m,X_{eye}=X_i,Y_{eye}=Y_i,i=i+1,n)$}
    \EndIf
     \If {$(S==4)$}
      \State {$X_i=X_{eye}-\frac{L}{2}+\frac{L}{2m}$}
      \State {$Y_i=Y_{eye}+\frac{L}{2}+\frac{L}{2m}$}
      \State {SearchPole$(S=1,L=\frac{L}{m},m,X_{eye}=X_i,Y_{eye}=Y_i,i=i+1,n)$}
    \EndIf

  \EndFunction
\end{algorithmic}
\end{algorithm}

\vspace{1cm}
\begin{theorem}\label{theoremB}

For any m-spiral, poles can be found at:

\begin{equation}\label{eq11}
  \begin{dcases}
    &X_{eye} = X_{ur}+L\frac{m-1}{m^{2}+1} = X_{lr}+L\frac{m-1}{m^{2}+1} \\[0.5em]
    &Y_{eye} = Y_{ur}-L\frac{m+1}{m^{2}+1} = Y_{lr}+L\frac{m(m-1)}{m^{2}+1}\\[0.5em]
  \end{dcases}
\end{equation}
\vspace{0.4cm}

\end{theorem}

\begin{proof}
The square centers converge to m-spiral pole when $i \to \infty$. So, equations \eqref{eq09} produce the limits
 
\begin{equation*}
  \begin{aligned}
    &\lim_{i \to \infty} k_1(i) \to 1 &  &\lim_{i \to \infty} k_2(i) \to 0
  \end{aligned}
\end{equation*}

Which gives

\begin{equation*}
  \begin{dcases}
    X_{eye} &= X_{ur}+L\left(\frac{m-1}{m^2+1}\right)\\[0.5em]    
    Y_{eye} &= Y_{ur}-L\left(\frac{m+1}{m^2+1}\right)\\    
  \end{dcases}
\end{equation*}
\vspace{0.4cm}

Since $L-L\frac{m+1}{m^2+1}=L\frac{m(m-1)}{m^2+1}$, we have by symmetry (see figure \ref{Fig5a}B)

\begin{equation*}
  \begin{dcases}
    X_{eye} &= X_{lr}+L\left(\frac{m-1}{m^2+1}\right)\\[0.5em]    
    Y_{eye} &= Y_{lr}+L\left(\frac{m(m-1)}{m^2+1}\right)\\    
  \end{dcases}
\end{equation*}
\vspace{0.4cm}
 
\end{proof}

\begin{corollary}\label{corolE}
It is immediate that

\begin{equation}
 \frac{Y_{eye}-Y_{lr}}{X_{eye}-X_{lr}}=m
\end{equation}
\end{corollary}
\vspace{0.4cm}

\begin{corollary}\label{corolF}
Some behaviors at the limit of m: if $m\to 1+$ then $(X_{eye}\to X_{lr},Y_{eye}\to Y_{lr})$; if $m\to \infty$ then $(X_{eye}\to X_{lr}=X_{ur},Y_{eye}\to (Y_{lr}+L)\to Y_{ur})$
\end{corollary}

\begin{figure}[h]
\centering
\includegraphics[width=1.0\textwidth]{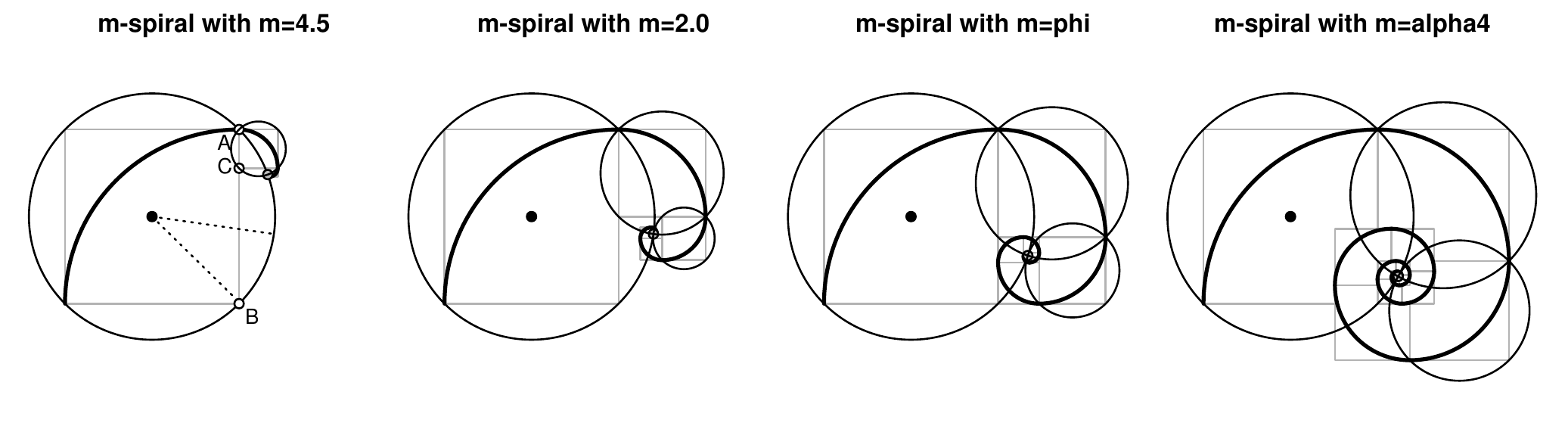}
\caption{Geometric examples of theorem \ref{theoremF}. Dotted lines represent the interval $2 \leq m < 1$}
\label{Fig8a}
\end{figure}
\vspace{0.4cm}

\begin{theorem}\label{theoremF}
  For any m-spiral, circumscribed circles around squares intercept the respective pole; i.e., for each $i^{th}$ square:
 
 \begin{equation}\label{eq34}
  (X_{eye}-X_{i})^2+(Y_{eye}-Y_{i})^2=\frac{L^2}{2m^{2i}}
 \end{equation}
 \vspace{0.4cm}
\end{theorem}

\begin{proof}
Considering the equations of lemma \ref{lemmaA} and theorem \ref{theoremB} we can rewrite equation \eqref{eq34} as follows:
  
  \begin{flalign*}
    &\left(X_{eye} - X_{i}\right)^2+\left(Y_{eye} - Y_{i}\right)^2=&
  \end{flalign*}
  
  \begin{flalign*}
    &=L^2 \left[\left(\frac{m-1}{m^2+1}-k_1(i)\frac{m-1}{m^2+1}+(-1)^ik_2(i)\right)^2+\left(-\frac{m+1}{m^2+1}+k_1(i)\frac{m-1}{m^2+1}+k_2(i)\right)^2\right]&\\[0.5em]
    &=L^2 \left[\left(\frac{m-1}{m^2+1}(1-k_1(i))+(-1)^ik_2(i)\right)^2+\left(\frac{m+1}{m^2+1}(1-k_1(i))-k_2(i)\right)^2\right]&
  \end{flalign*}
  \vspace{0.4cm}

 Simplifying with $k_1^{'}(i)=1-k_1(i)$ and squaring:
  
 \footnotesize
  \begin{flalign*}
    &=L^2 \left[\frac{(m-1)^2}{(m^2+1)^2}k_1^{'}(i)^2+2(-1)^{i}k_1^{'}(i)k_2(i)\frac{m-1}{m^2+1}+k_2(i)^2+\frac{(m+1)^2}{(m^2+1)^2}k_1^{'}(i)^2-2k_1^{'}(i)k_2(i)\frac{m+1}{m^2+1}+k_2(i)^2\right]&
  \end{flalign*}
  \vspace{0.4cm}
  
  \normalsize
  \begin{flalign*}
    &=2L^2 \left[\frac{k_1^{'}(i)^2+k_1^{'}(i)k_2(i)((-1)^{i}(m-1)-(m+1))+k_2(i)^2(m^2+1)}{m^2+1}\right]&
  \end{flalign*}
  \vspace{0.4cm}
  
  We now need to consider the parity. So, for even i, $(i=2j)$:
  
  \begin{flalign*}
    &=2L^2 \left[\frac{\frac{1}{m^{4j}}+\frac{(-1)^{j}(-1)^{j}}{2m^{4j}}(-2)+\frac{1}{4m^{4j}}(m^2+1)}{m^2+1}\right]=\frac{L^2}{2m^{4j}}=\frac{L^2}{2m^{2i}}&
  \end{flalign*}
  \vspace{0.4cm}
  
  for odd i, $(i=2j+1)$:
  
  \begin{flalign*}
    &=2L^2 \left[\frac{\frac{1}{m^{4j}}+\frac{(-1)^{j}(-1)^{j}}{2m^{4j+1}}(-2m)+\frac{1}{4m^{4j+2}}(m^2+1)}{m^2+1}\right]=\frac{L^2}{2m^{4j+2}}=\frac{L^2}{2m^{2i}}&
  \end{flalign*}
  \vspace{0.4cm}

 \end{proof}

\begin{corollary}\label{corolG}
If m-spirals share the same first square of side $L$ than all poles lie on circle circumscribed around it, i.e.: $(X_{eye}-X_{0})^2+(Y_{eye}-Y_{0})^2=\frac{L^2}{2}$. (Note the independence of $m$).
\end{corollary}
\vspace{0.4cm}

\section{Notes on diagonals}

To continue from here, it is more convenient to put equations of theorem \ref{theoremB} in a vector form. Without loss of generality, we will set the vertex $E=(X_{lr},Y_{lr})$ as the origin of our standard basis vector. 

\begin{definitions}\label{definC}
 Keep in mind figure \ref{Fig7a}: let ${A,B,C,D,E,}$ be vertices; let $d_1$ and $d_2$ be diagonals $\overline{AC}$ and $\overline{BD}$, respectively; let $P=(X_{eye},Y_{eye})$ be the pole; let $\mathbf{v_{eye}}$ and $\mathbf{u_{eye}}$ be vectors, such that $(\mathbf{v_{eye}} \bot \mathbf{u_{eye}})$.
\end{definitions}

\begin{figure}[h]
\centering
\includegraphics[width=0.8\textwidth]{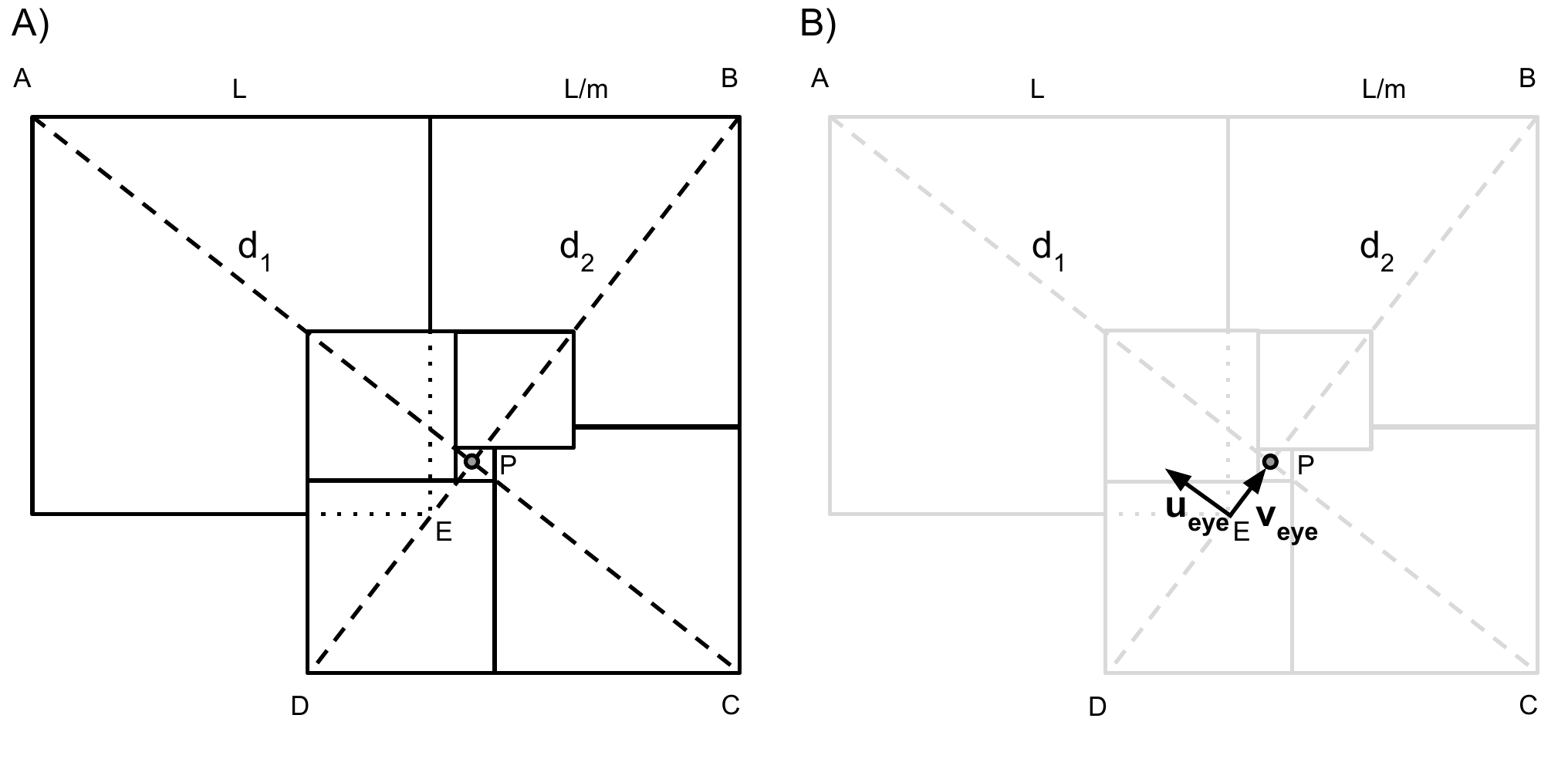}
\caption{Partial representation of a m-spiral. A) Diagonals of extreme square vertices ${A,B,C,D}$. B) Geometric visualization of vectors $\mathbf{v_{eye}}$ and $\mathbf{u_{eye}}$}
\label{Fig7a}
\end{figure}
\vspace{0.4cm}

\begin{theorem}\label{theoremC}
  For any m-spiral, diagonals $d_1$ and $d_2$ are orthogonals, i.e. $(d1\perp d2)$. 
\end{theorem}

\begin{proof}

Given equations \eqref{eq11}, we can conceive the vector $\mathbf{v_{eye}}$ as:

\begin{equation}\label{eq19}
  \begin{aligned}
    \mathbf{v_{eye}}&= \left\langle X_{eye}-X_{lr},Y_{eye}-Y_{lr} \right\rangle = \left\langle L\left(\frac{m-1}{m^2+1}\right),L\left(\frac{m(m-1)}{m^2+1}\right)\right\rangle = L 
    \begin{bmatrix}
      \frac{m-1}{m^2+1} \\[0.5em]
      \frac{m(m-1)}{m^2+1}
    \end{bmatrix} \\     
  \end{aligned}
\end{equation}
\vspace{0.4cm}

 We have the following vertex coordinates:

 \begin{equation}\label{eq21}
    \begin{array}[b]{l}
     A = (-L,L)\\[0.5em]
     B = \left(\frac{L}{m},L\right)\\[0.5em]
     C = \left(\frac{L}{m},L-\frac{L}{m}-\frac{L}{m^2}\right)\\[0.5em]
     D = \left(\frac{L}{m}-\frac{L}{m^2}-\frac{L}{m^3},L-\frac{L}{m}-\frac{L}{m^2}\right)
     \end{array}
\end{equation}
\vspace{0.4cm}

and we can define the vectors:

\begin{equation}\label{eq22}
  \begin{aligned}
  \mathbf{d1}=\overrightarrow{AC}=
    \begin{bmatrix}
      -L-\frac{L}{m} \\[0.5em]
      \frac{L}{m}+\frac{L}{m^2}
    \end{bmatrix}&,&
  \mathbf{d2}=\overrightarrow{BD}=
    \begin{bmatrix}
      \frac{L}{m^2}+\frac{L}{m^3} \\[0.5em]
      \frac{L}{m}+\frac{L}{m^2}
    \end{bmatrix}
  \end{aligned}
\end{equation}
\vspace{0.4cm}

If we rotate at 90\textdegree (on page plane) the vector $\mathbf{d2}$, than we can conceive one new vector $\mathbf{d3}$, such that $\mathbf{d3}\perp \mathbf{d2}$. 

 \begin{equation}\label{eq24}
 \mathbf{d3}=
   \begin{bmatrix}
    0 & -1 \\[0.5em]
    1 & 0
  \end{bmatrix}
  \cdot
  \begin{bmatrix}
    \frac{L}{m^2}+\frac{L}{m^3} \\[0.5em]
    \frac{L}{m}+\frac{L}{m^2}
  \end{bmatrix}
  =
  \begin{bmatrix}
    -\frac{L}{m}-\frac{L}{m^2} \\[0.5em]
    \frac{L}{m^2}+\frac{L}{m^3}
  \end{bmatrix}
\end{equation}
\vspace{0.4cm}

Finally, we need to prove that $\mathbf{d3}\parallel \mathbf{d1}$. One of the ways is demonstrate that there is a scalar $t$, such that $\mathbf{d3} = t\mathbf{d1}$. We can see that $t=\frac{1}{m}$ satisfies this condition.

\end{proof}

\begin{theorem}\label{theoremD}
For any m-spiral, diagonals $d_1$ and $d_2$ intercept the pole $(X_{eye},Y_{eye})$.
\end{theorem}

\begin{proof}
 If $(\mathbf{u_{eye}} \bot \mathbf{v_{eye}})$ than 

 \begin{equation}\label{eq25}
  \mathbf{u_{eye}} = L
  \begin{bmatrix}
    \frac{-m(m-1)}{m^2+1} \\[0.5em]
    \frac{m-1}{m^2+1}
  \end{bmatrix}
\end{equation}
 \vspace{0.4cm}
 
We can create the following line equations in vector form (see figure \ref{Fig7a}B):

 \begin{equation}\label{eq26}
   \begin{aligned}
      &\mathbf{r_1}=\mathbf{v_{eye}}+k \mathbf{u_{eye}},&
      &\mathbf{r_2}=t \mathbf{v_{eye}}
   \end{aligned}
 \end{equation}
 \vspace{0.4cm}
 
It is easy to see that $\mathbf{r_1}$ and $\mathbf{r_2}$ will intercept the point $(X_{eye},Y_{eye})$ when $k=0$ and $t=1$, respectively. Now, we have to show that line equations $\mathbf{r_1}$ and $\mathbf{r_2}$ intercept points ${A,C}$ and ${B,D}$, respectively. 

For simplicity, let us make

  \begin{equation}\label{eq27}
    \begin{aligned}
      &a =\left(\frac{m-1}{m^2+1}\right),&&
      &b =\left(1-\frac{1}{m}-\frac{1}{m^2}\right) 
    \end{aligned}
  \end{equation}
  \vspace{0.4cm}

and rewrite vectors \eqref{eq19} and \eqref{eq25} respectively as
  
  \begin{equation}\label{eq28}
  \begin{aligned}
    \mathbf{v_{eye}} = L
    \begin{bmatrix}
      a \\[0.5em] 
      ma 
    \end{bmatrix},&&
    \mathbf{u_{eye}} = L
    \begin{bmatrix}
      -ma \\[0.5em] 
      a
    \end{bmatrix}
  \end{aligned}
  \end{equation}
  \vspace{0.4cm}
  
We have to prove that $\mathbf{r_1}$ intercepts $A$, finding a scalar $t$ such that 

  \begin{equation}\label{eq29}
    \begin{bmatrix}
      -L \\[0.5em] 
      L 
    \end{bmatrix} = L
    \begin{bmatrix}
      a \\[0.5em] 
      ma 
    \end{bmatrix} + tL 
    \begin{bmatrix}
      -ma \\[0.5em] 
      a 
    \end{bmatrix}
  \end{equation}
  \vspace{0.4cm}

We see that $t=\frac{m+1}{m-1}$ satisfies \eqref{eq29}.

Now, we must prove that $\mathbf{r_1}$ intercepts $C$, also finding a scalar $t$ such that

  \begin{equation}\label{eq30}
    \begin{bmatrix}
      \frac{L}{m} \\[0.5em] 
      Lb 
    \end{bmatrix} = L
    \begin{bmatrix}
      a \\[0.5em] 
      ma 
    \end{bmatrix} + tL 
    \begin{bmatrix}
      -ma \\[0.5em] 
      a 
    \end{bmatrix}
  \end{equation}
  \vspace{0.4cm}

We see that $t=\frac{b-1}{m-1}$ satisfies \eqref{eq30}. So, we have proved that $\mathbf{r_1}$ coincides with diagonal $d_1$ and pass through pole $(X_{eye},Y_{eye})$.

The same we have to do with $\mathbf{r_2}$: that it intercepts $B$, finding a scalar $t$ such that 

  \begin{equation}\label{eq31}
    \begin{bmatrix}
      \frac{L}{m} \\[0.5em] 
      L 
    \end{bmatrix} = tL
    \begin{bmatrix}
      a \\[0.5em] 
      ma 
    \end{bmatrix}
  \end{equation}
  \vspace{0.4cm}

We see that $t=\frac{1}{ma}$ satisfies \eqref{eq31}.

And we must prove that $\mathbf{r_2}$ intercepts $D$, also finding a scalar $t$ such that

  \begin{equation}\label{eq32}
    \begin{bmatrix}
      \frac{L}{m}b \\[0.5em] 
      b 
    \end{bmatrix} = tL
    \begin{bmatrix}
      a \\[0.5em] 
      ma 
    \end{bmatrix}
  \end{equation}
  \vspace{0.4cm}

We see that $t=\frac{b}{ma}$ satisfies \eqref{eq32}. So, we have proved that $\mathbf{r_2}$ coincides with diagonal $d_2$ and pass through pole $(X_{eye},Y_{eye})$.
  
 \end{proof}

\begin{corollary}\label{corolH}
 From any m-spiral, diagonals $d_1$ and $d_2$ have inclinations $-\frac{1}{m}$ and $m$, respectively. (directly from corollary \ref{corolE} and from $d_1 \bot d_2$).
\end{corollary}

\begin{theorem}\label{theoremE}
 For any m-spiral, the ratio between the length of diagonals $d_1$ and $d_2$ is $m$, i.e.:
 
 \begin{equation}\label{eq33}
   \frac{\|\mathbf{d_1}\|}{\|\mathbf{d_2}\|}=m
 \end{equation}
 \vspace{0.4cm}

\end{theorem}

\begin{proof}
 Given vector definitions \eqref{eq22}, we have:
 \begin{equation*}
  \begin{aligned}
  \frac{\|\mathbf{d_1}\|}{\|\mathbf{d_2}\|} &= \frac{\sqrt{\left(-L-\frac{L}{m}\right)^2+\left(\frac{L}{m}+\frac{L}{m^2}\right)^2}}{\sqrt{\left(\frac{L}{m^2}+\frac{L}{m^3}\right)^2+\left(\frac{L}{m}+\frac{L}{m^2}\right)^2}}\\[0.5em]
  \frac{\|\mathbf{d_1}\|}{\|\mathbf{d_2}\|} &= \frac{\sqrt{\frac{L^2(m+1)^2}{m^2}+\frac{L^2(m+1)^2}{m^4}}}{\sqrt{\frac{L^2(m+1)^2}{m^6}+\frac{L^2(m+1)^2}{m^4}}}\\[0.5em]
  \frac{\|\mathbf{d_1}\|}{\|\mathbf{d_2}\|} &= \frac{\frac{L}{m^2}\sqrt{\left(m+1\right)^2\left(m^2+1\right)}}{\frac{L}{m^3}\sqrt{\left(m+1\right)^2\left(m^2+1\right)}}\\[0.5em]
  \frac{\|\mathbf{d_1}\|}{\|\mathbf{d_2}\|} &= m
  \end{aligned}
 \end{equation*}
 \vspace{0.1cm}
 
\end{proof}

\section{Conclusions}

Some examples of theorem \ref{theoremF} can be seen in figure \ref{Fig8a}. This theorem may be a natural consequence from self-similarity in logarithmic-like spirals. In our context, this implies that properties related to one square will be inherited by the remaining. Thereby, if circumscribed circle around one square intercepts the pole then all the others will do.

Dotted lines in figure \ref{Fig8a} represent the interval of $2 \leq m < 1$. We noticed an asymmetry: although this interval comprises half of first square side length, it does not correspond to half of the arc covering this side. Furthermore, the tendency to upper right vertex as $m \to \infty$ is very fast: with only $m=60$ the pole is already near this convergent point. We should also remark that $m$ is not well-defined for 1. When $m=1$ the whirling squares do not decrease and the arcs do not form a spiral but they close themselves in a circle (review figure \ref{Fig4a}).   

\begin{wrapfigure}{r}{0.4\textwidth}
  \vspace{-20pt}
  \captionsetup{font=scriptsize}
  \caption*{Image free: Chris 73 at Wikimedia Commons}
  \begin{center}
    \includegraphics[width=0.3\textwidth]{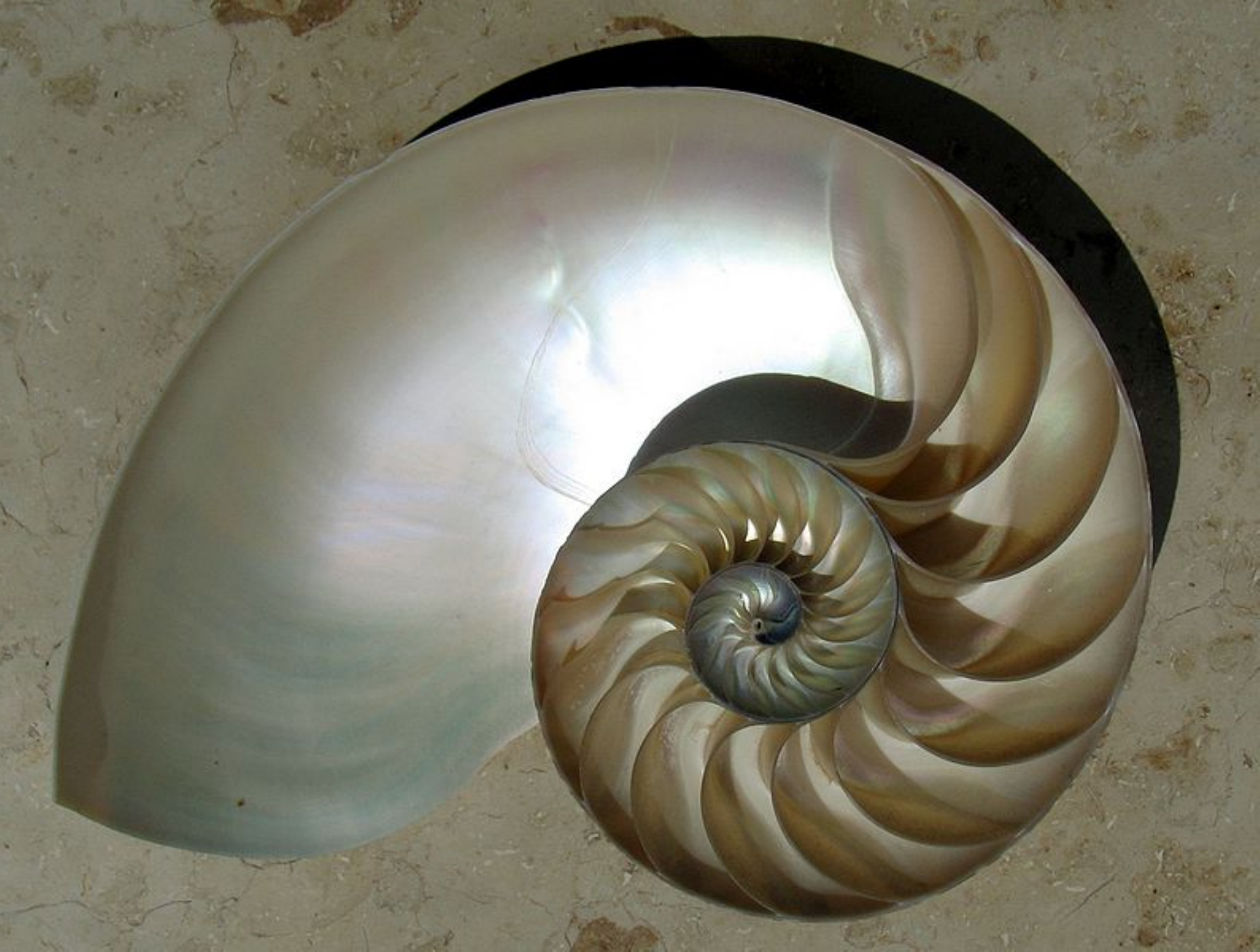}
  \end{center}
  \vspace{-2pt}
  \captionsetup{font=small}
  \caption[text]{Nautilus shell spiral.}
  \vspace{-10pt}
  \label{Fig9a}
\end{wrapfigure}

The classical golden constructors are not alone. There is a whole family (or families) of other amazing sections and spirals \cite{Stakhov_a,Fowler_a,Falbo_a,Spinadel_a}. We see here that many of the fascinating properties attributed to golden rectangle and its spiral can be extended to generalized m-spirals. Our obsession with $\phi$ should be reviewed. In this sense, there is a growing debate about what is the real scope of $\phi$-based mathematics in nature and arts \cite{Sharp_a,Stakhov_b,Cooke_a,Eydt_a,Shipman_a,Markowsky_a,Fonseca_a}. One frequently cited example is the spiral shell of cephalopod \emph{Nautilus} spp. We see in figure \ref{Fig9a} that better $m$ values may exist to fit the Nautilus spiral than $m=\phi$. Indeed, Falbo \cite{Falbo_a} found an average $m=1.33$ (close to Stakhov's 4-Fibonacci number - $\alpha_4=1.324...$; see figure \ref{Fig4a}) and a min-max interval $[1.24,1.43]$ that does not cover $\phi$. 

Allow me to finish with a short digression. The mathematician Clifford Pickover poetically called \emph{eye of god} the pole of a Fibonacci spiral, due to ``divine'' properties historically attributed to golden ratio \cite{Pickover_a}. We may have extended this metaphor, demonstrating that the \emph{eyes of god} are infinite, that all of them are intercepted by infinite circles and lie on a ``special'' circle.

\medskip

\noindent MSC2010: 51M99, 11Z99

\end{document}